\documentclass{amsart}

\usepackage[all]{xy}
\usepackage[latin1]{inputenc}
\usepackage{hyperref}
\usepackage{amssymb}
\usepackage{tensor}
\usepackage{nicefrac}

\newtheorem*{Maintheorem*}{Main Theorem}
\newtheorem*{theorem*}{Theorem}

\newtheorem{theorem}{Theorem}

\newtheorem{lemma}{Lemma}
\newtheorem{remark}{Remark}
\newtheorem{proposition}{Proposition}
\newtheorem{corollary}{Corollary}
\newtheorem*{corollary*}{Corollary}

\newcommand{\Oh}{\mathcal{O}}
\newcommand{\C}{\mathbb{C}}
\newcommand{\R}{\mathbb{R}}

\newcommand{\PR}{\mathbb{P}}

\newcommand{\la}{\alpha}
\newcommand{\lb}{\overline{\beta}}
\newcommand{\ld}{\overline{\delta}}
\newcommand{\lc}{\gamma}
\newcommand{\ls}{\sigma}
\newcommand{\lr}{\ovl{\tau}}

\newcommand{\lam}{\lambda}

\newcommand{\tr}{\operatorname{tr}}

\newcommand{\ovl}{\overline}
\newcommand{\dbar}{\bar \partial}
\newcommand{\dl}{ \partial}

\renewcommand\>{\rangle}

\newcommand{\X}{\mathcal{X}}

\newcommand{\jbar}{\ovl{\jmath}}

\newcommand{\bbar}{\ovl{b}}
\newcommand{\lbar}{\ovl{l}}

\newcommand{\sbar}{\ovl{s}}
\newcommand{\labar}{\ovl{\alpha}}
\newcommand{\lbbar}{\ovl{\beta}}

\newcommand{\Ric}{\operatorname{Ric}}
\newcommand{\Aut}{\operatorname{Aut}}
\newcommand{\MA}{\operatorname{MA}}

\begin{document}

\title{An approach to Griffiths conjecture}

\author{Philipp Naumann}
\address{Philipp Naumann, Université de Grenoble I,
Institut Fourier, 38402 Saint-Martin d'Hères, France}
\email{Philipp.Naumann@univ-grenov-alpes.fr}

\thanks{}



\date{}

\begin{abstract}
The Griffiths conjecture asserts that every ample vector bundle $E$ over a compact complex manifold $S$ admits a hermitian metric with positive curvature in the sense of Griffiths. In this article we give a sufficient condition for a positive hermitian metric on $\Oh_{\PR(E^*)}(1)$ to induce a Griffiths positive $L^2$-metric on the vector bundle $E$. This result suggests to study the relative K\"ahler-Ricci flow on $\Oh_{\PR(E^*)}(1)$ for the fibration $\PR(E^*)\to S$. We define a flow and give arguments for the convergence.
\end{abstract}

\maketitle

\section{Introduction}
Let $E \to S$ be a holomorphic vector bundle of rank $r$ over a compact complex manifold $S$. The projectivized bundle 
$\X:=\PR(E^*)$ carries the tautological line bundle $\Oh_{\PR(E^*)}(1)$ which we also denote by $\Oh_E(1)$. There is an isomorphism $f_*(\Oh(1)) \cong E$ where $f: \X \to S$ is the projection map. 
In \cite{Ha66} Hartshorne defined the ampleness of a vector bundle over a projective manifold: A vector bundle is called ample if the tautological line bundle $\Oh_E(1)$ is ample over $\X$. On the other hand, Griffiths defined in \cite{Gr69} a notion of positivity for a hermitian holomorphic vector bundle $(E,H)$ by using the curvature of $(E,H)$ which is now called Griffiths positivity. For a line bundle the ampleness of the bundle is equivalent to Griffiths positivity. It is also well-known that a Griffiths positive metric on $E$ induces a positive metric on $\Oh_E(1)$, hence any Griffiths positive vector bundle is always ample. Griffiths conjectured also the converse, that an ample vector bundle $E$ also carries a Griffiths positive metric. Umemura \cite{Um73} as well as Campana and Flenner \cite{CP90} gave an affirmative answer to this question when the base $S$ is a curve. But in general finding a Griffiths positive metric on an ample vector bundle seems to be very difficult and is worth to be investigated. 

Another important contribution which points into the direction of this conjecture comes from the general positivity theory for direct images of adjoint bundles.  Berndtsson proved in \cite{Be09} that the bundle $E \otimes \det E$ has always a hermitian metric which is Nakano positive for any given ample vector bundle $E$. Together with the result of Demailly and Skoda \cite{DS79} which says that $(E\otimes \det E, H\otimes \det H)$ is Nakano positive if $(E,H)$ is Griffiths positive, this can be seen as a further indication for the Griffiths conjecture. The Nakano positivity of $E \otimes \det E$ follows from Berndtsson's main theorem by the identity $f_*(K_{\X/S} \otimes \Oh(r+1)) \cong E \otimes \det E$. This article is motivated by his result. In order to study $E$ instead of $E  \otimes \det E$ we look at $f_*(\Oh(1)) \cong E$ which is no longer the direct image of an adjoint bundle (an ample twisted relative canonical bundle). 
\newpage
Now assume that $E$ is ample and choose a positive hermitian metric $h$ on $\Oh_E(1)$. We define the positive form
$$
\omega_{\X}:= -\sqrt{-1}\dl\dbar\log h.
$$
This gives K\"ahler forms 
$$
\omega_s:=\omega_{\X}|_{X_s}
$$ 
on the fibers $X_s$ which induce a hermitian metric on $\det T_{\X/S}=K_{\X/S}^{-1}$ denoted by 
$$
(- \sqrt{-1}\dl \dbar \log h)^n,
$$
where $n=\dim X_s=r-1$.
We define the $L^2$-inner product of two sections $u,v \in H^0(X_s,\Oh_{\PR(E_s^*)}(1)) \cong E_s$ on the direct image $f_*(\Oh(1)) \cong E$ at a point $s \in S$ by
\begin{equation}
\label{L2}
\<u,v\>_{L^2}(s) := \int_{X_s}{h(u,v)\, \frac{\omega_s^n}{n!}}.   
\end{equation}
When the hermitian metric $h$ is induced from a hermitian metric $H$ on $E$, we already know from \cite[Theorem 7.1]{LSY13} that the $L^2$-metric then gives back the original metric $H$ on $E$ up to a constant factor. In particular we know that if we start with a Griffiths positive metric $H$ on $E$, it should be possible to prove the Griffiths positivity of the $L^2$-metric directly by looking at its curvature expression. By analyzing this situation very carefully and using a general curvature formula for direct images due to To and Weng \cite{TW03}, we are able to prove the following result:
\begin{theorem}
\label{thm}
If the canonical isomorphism
\begin{equation}
\label{iso}
K_{\X/S}^{-1}\cong \Oh_E(r) \otimes f^*\det(E)^{-1}
\end{equation}
becomes an isometry for some hermitian metric on $f^*\det(E)^{-1}$ that is the inverse of the pullback of a hermitian metric $G$ on $\det E$, then the $L^2$-metric on $f_*(\Oh(1))=E$ is Griffiths positive. 
\end{theorem} 

We give the proof in the next section. 
The relation (\ref{iso}) stated in Theorem \ref{thm} implies that the curvature form of $h$ defines K\"ahler-Einstein metrics on the fibers, which are projective spaces. But the Fubini-Study metrics on $\PR(E^*_s)$ are in one-to-one correspondence with hermitian structures on  $E^*_s$, hence the theorem should actually give a characterisation of those metrics on $\Oh_E(1)$ which are induced by metrics on $E$. This gives a link between the K\"ahler-Einstein problem on projective spaces and the Griffiths conjecture. Therefore we propose in section 3 to study the relative K\"ahler-Ricci flow on the bundle $\Oh_E(r)$ instead of $K_{\X/S}^{-1}$. Here we rely on Berman's article \cite{Ber13} and first recall his definition for the normalized relative K\"ahler-Ricci flow on the level of hermitian metrics for the bundle $K_{\X/S}^{-1}$. By using more recent results for the K\"ahler-Ricci flow on Fano K\"ahler-Einstein manifolds, we explain briefly how to extend his result on the convergence of the flow to the case of non-discrete automorphism groups of the fibers. Afterwards we define a flow on $\Oh_E(r)$ and give arguments for its convergence.  Finally we state the evolution equation for the geodesic curvature function of the evolving metrics on $\Oh_E(r)$ which completely encodes its positivity. 
\vfill
{\bf Acknowledgements } The author would like to thank Tristan Collins and Sébastien Boucksom for discussions about the convergence of the K\"ahler-Ricci flow in the Fano case.

\section{Griffiths positivity of $L^2$-metrics}
In this section we prove Theorem \ref{thm}. First we introduce the main objects which appear along the computation of the curvature. We cite the formula of To and Weng and prove the Propositions \ref{harm} and \ref{elleq} that allow to simplify the curvature expression. It will also clarify the structure of the geodesic curvature function that finally allows to evaluate the curvature integrals.
 
We use local holomorphic coordinates $(s^k)$ on the base $S$ and coordinates $(z^{\la})$ on the fibers and write
$$
\omega_{\X}= \sqrt{-1}\left( g_{\la\lb}\, dz^{\la} \wedge dz^{\lb} + h_{k\lb}\, ds^{k}\wedge dz^{\lb} 
+ h_{\la \lbar}\, dz^{\la} \wedge ds^{\lbar} + h_{k\lbar}\, ds^k \wedge ds^{\lbar} \right)
$$ 
Thus the K\"ahler forms are given by 
$$
\omega_s:= \sqrt{-1}\, g_{\la\lb}\, dz^{\la} \wedge d z^{\lb}
$$ 
and the induced metric on  $K_{\X/S}^{-1}$ can be written as
$$
\det (g_{\la\lb}).
$$
According to \cite{Sch93} we denote the horizontal lift of a tangent vector $\dl_k=\dl/\dl s^k$ on the base $S$ by $v_k$. It is given by
$$
v_k = \dl_k + a_k^{\la}\,\dl_{\la}
$$
where $\dl^{\la}=\dl/\dl z^{\la}$ and 
$$
a_k^{\la}=-g^{\lb\la}\, h_{k\lb}.
$$
We obtain the Kodaira-Spencer forms by
$$
A_k := \dbar(v_k)|_{X_s}
$$
and define the \emph{geodesic curvature} in the direction of $k,l$ by
$$
c(\varphi)_{k\lbar} = \<v_k,v_l\>_{\omega_{\X}}. 
$$
Here we differ slightly in notation from \cite{Sch93} and use instead the notation from \cite{Be11} to indicate that $c(\varphi)$ depends on $h$ where locally $h=e^{-\varphi}$. 
In our local coordinates we have
$$
c(\varphi)_{k\lbar} = h_{k\lbar} - a_k^{\ls}a_{\lbar\ls}.
$$
We observe that the matrix $(c(\varphi)_{k\lbar})$ is positive definite iff the hermitian line bundle $(\Oh_E(1),h)$ is positive which was our assumption.

The main ingredient in the proof of Theorem \ref{thm} is a general curvature formula for the direct images of the form $p_*L$ due to To and Weng: 
\begin{theorem*}[\cite{TW03}]
Let $p: \X \to S$ be a smooth family of $n$-dimensional compact complex manifolds and $(L,h) \to \X$ be a hermitian holomorphic line bundle. Furthermore assume that $\X$ admits a smooth $(1,1)$-form $\omega_{\X}$ such that its restrictions $\omega_{s}:=\omega_{\X}|_{X_s}$ are K\"ahler forms on $X_s$ and such that 
$$
c_1(L,h)=\frac{k}{2 \pi} \omega_{\X}
$$   
for some $k \in \R$. Assume that $p_*L$ is locally free on $S$ with fiber $(p_*L)s=H^0(X_s,L_s)$. Then the curvature tensor $\Theta$ of the associated $L^2$-metric (defined as in (\ref{L2})) on $p_*L$ is given by
\begin{eqnarray*}
\Theta_{a\bbar i \jbar}(s) = &-& \int_{X_s}{\<G(A_{i\lbbar}^{\lc}t_{a;\lc}dz^{\lbbar}), A_{\jbar \lbbar}^{\lc}t_{b;\lc}dz^{\lbbar} \>_h
\frac{\omega_s^n}{n!}}\\
&+& \int_{X_s}{(kc(\varphi)_{i\jbar} + \Box_{\omega_s} c(\varphi)_{i\jbar})\<t_a,t_b\>_h\frac{\omega_s^n}{n!}}
\end{eqnarray*} 
for $t_a,t_b \in H^0(X_s,L_s)$. Here $G$ denotes the Green's operator.
\end{theorem*}
The formula in this general form is of course difficult to deal with. We prove the following two results (cf. \cite{Sch93}) that allow considerable simplifications in our setting: 
\begin{proposition}
\label{harm}
Under the assumption of Theorem \ref{thm}, the Kodaira-Spencer forms $A_i$ are harmonic, hence zero.
\end{proposition}
\begin{proof}
We use the symbol $;$ to denote covariant derivatives. 
On the fiber $X_s$ we have
\begin{eqnarray*}
g^{\ld\lc}A_{i\lb \ld;\lc} &=& g^{\ld \lc}a_{i\ld;\lb\lc} = g^{\ld \lc}a_{i\ld;\lc\lb} - g^{\ld \lc} a_{i\lr}R^{\lr}_{\ld\;\lb \lc}\\
&=& \left(g^{\ld \lc}\dl_{\lc}\left(\frac{\dl^2 \log h}{\dl s^{i} \dl z^{\ld}}\right)\right)_{;\lb} + a_{i\lr}R^{\lr}_{\lb}\\
&=& \left(g^{\ld \lc}\dl_i\left(\frac{\dl^2 \log h}{\dl s^{\lc} \dl z^{\ld}} \right) \right)_{;\lb}+ a_{i\lr} g^{\lr \la} R_{\la\lb}\\
&=& \left(-g^{\ld \lc} \dl_i(g_{\lc \ld}) \right)_{;\lb} + a_i^{\la}(n+1)g_{\la \lb} \\
&=& -\dl_{\lb}\dl_i\log \det(g) + (n+1)a_{i\lb} \\
&=& -\frac{\dl^2 \log \det(g)}{\dl s^i \dl z^{\lb}} + (n+1)\frac{\dl^2 \log h}{\dl s^i \dl z^{\lb}}\\
&=& -\frac{\dl^2 \log (h^{n+1}\cdot f^*G^{-1})}{\dl s^i \dl z^{\lb}} + \frac{\dl^2 \log h^{n+1}}{\dl s^i \dl z^{\lb}}\\
&=& -\frac{\dl^2 \log h^{n+1}}{\dl s^i \dl z^{\lb}} + \frac{\dl^2 \log h^{n+1}}{\dl s^i \dl z^{\lb}}=0,
\end{eqnarray*}
because by our assumption $(\det T_{\mathcal{X}/S},\det(g)) \cong (\mathcal{O}(n+1)\otimes f^*(\det E^*),h^{n+1}\cdot (f^*G)^{-1})$ and the pullback metric $f^*G^*$ is constant on the fibers.
\end{proof}
\begin{proposition}
\label{elleq}
Under the assumption of Theorem \ref{thm}, the following equation holds:
\begin{equation}
\label{ell}
(\Box_{\omega_s} -r)c(\varphi)_{k\lbar} = A_k\cdot A_{\lbar} - f^*(R^{\det E}_G)_{k\lbar},
\end{equation}
where $(R^{\det E}_{G})_{k\lbar}$ is the curvature of $(\det E,G)$ in the direction of $\dl_k$ and $\dl_l$. 
\end{proposition}
\begin{proof}
It holds $-\Box(v_k \cdot v_{\lbar})=g^{\ld\lc}(h_{k\lbar} - a_k^{\ls}a_{\lbar\ls})_{;\lc\ld}.$ We compute
\begin{eqnarray*}
g^{\ld\lc}h_{k\lbar;\lc\ld} &=& g^{\ld\lc}\dl_k\dl_{\lbar}h_{\lc \ld}\\
&=& g^{\ld\lc}\dl_k\dl_{\lbar}g_{\lc \ld}\\
&=&\dl_k(g^{\ld\lc}\dl_{\lbar}g_{\lc\ld})-\dl_kg^{\ld \lc}\dl_{\lbar}g_{\lc\ld}\\
&=&\dl_k(g^{\ld\lc}\dl_{\lbar}g_{\lc\ld})+g^{\ld \la}g^{\lb\lc}\dl_kg_{\la \lb}\dl_{\lbar}g_{\lc\ld}\\
&=& \dl_k\dl_{\lbar} \log \det(g_{\lc\ld}) + a_k^{\lc;\ld}a_{\lbar\lc;\ld}\\
&=& \dl_k\dl_{\lbar}\log(h^{n+1}\cdot (f^*G)^{-1}) + a_{k;\lc}^{\ls}a_{\lbar\ls;\ld}g^{\ld\lc}\\
&=& -(n+1)h_{k\lbar} + f^*(R_G^{\det E})_{k\lbar} + a_{k;\lc}^{\ls}a_{\lbar\ls;\ld}g^{\ld\lc}
\end{eqnarray*}
and
$$
(a_k^{\ls}a_{\lbar \ls})_{\lc\ld}g^{\ld\lc} = (a_{k;\lc\ld}^{\ls}a_{\lbar\ls} + A_{k_{\ld}^{\ls}}A_{\lbar\ls\lc} + a_{k;\lc}^{\ls}a_{\lbar\ls;\ld} + a_k^{\ls}A_{\lbar\ls\lc;\ld})g^{\ld\lc}.
$$
The last term vanishes because of the harmonicity of $A_l$ and 
\begin{eqnarray*}
a_{k;\lc\ld}^{\ls}g^{\ld\lc} &=& A_{k\ld;\lc}^{\ls}g^{\ld\lc} + a_k^{\lam}R^{\ls}_{\lam\ld\lc}g^{\ld\lc}\\
&=& 0 - a_k^{\lam}R^{\ls}_{\lam}\\
&=& - (n+1)a_k^{\ls}.
\end{eqnarray*}
\end{proof}
We observe that by integrating the equation (\ref{ell}) against the volume form induced by $\omega_{s}$ we get 
\begin{corollary}
\begin{equation}
\label{detcur}
\frac{f^*(R^{\det E}_G)_{k\lbar}}{(r-1)!} = \int_{X_s}{r \,c(\varphi)_{k\lbar}\;\frac{\omega_s^n}{n!}} \quad \mbox{if} \quad \int_{X_s}{\omega_s^n} =1.
\end{equation}
\end{corollary}
Now we can prove Theorem \ref{thm}: 
\begin{proof}

First we can make some reductions: By placing one-dimensional discs in all directions in the base $S$, we can restrict to the case of a one-dimensional base $S$. We fix a point $s \in S$ and consider the corresponding K\"ahler-Einstein form $\omega_s$. Because this is a Fubini-Study form, we can choose a local holomorphic frame $e_1,\ldots,e_r$ of $E^*$ around $s$ and the corresponding coordinates homogenous coordinates $W_1,\ldots, W_r$ on $\PR(E_s^*)$  such that this form becomes the standard Fubini-Study form on $\PR(E_s^*)\cong \PR^n$ which we denote by 
$\omega_{FS}$. At the same time we view the coordinates $W_1,\ldots,W_r$ as a base of global holomorphic sections of $f_*(\Oh_E(1)|_{X_s})\cong E_s$. Because the hermitian metric $h|_{X_s}$ and the standard hermitian metric on $\Oh_E(1)|_{X_s}$ given by 
$$
\frac{1}{|W^2|} := \frac{1}{|W_1|^2 + |W_2|^2 + \ldots + |W_r|^2}
$$
have the same curvature form on $X_s=\PR(E_s^*)$, hence they diver by a positive constant $\delta$ on the fixed fiber $X_s$.

Now we apply the formula of To and Weng to study the curvature of the $L^2$-metric given by $(\ref{L2})$ on $f_*\Oh_{E}(1)\cong E$. Using Proposition \ref{harm} the first term in the curvature formula completely disappears. Using the further reductions we made the second term reads as
$$
\Theta_{i\jbar s\sbar}(s) = \int_{X_s}{(c(\varphi) + \Box_{\omega_{FS}}c(\varphi)) \delta \frac{W_i W_{\jbar}}{|W|^2}\; \frac{\omega_{FS}^n}{n!}}.
$$
We invoke also Proposition \ref{elleq} to obtain
$$
\Theta_{i\jbar s\sbar}(s) = \int_{X_s}{((r+1)c(\varphi) - f^*R^{\det E}_G) \delta \frac{W_i W_{\jbar}}{|W|^2}\; \frac{\omega_{FS}^n}{n!}}.
$$
Now we remind ourselves that we are integrating over $\PR^n$ with respect to the standard Fubini-Study form. For such integrals we have the following (see \cite[Lemma 4.1]{LSY13})
\begin{lemma}
$$
\int_{\mathbb{P}^n}{\frac{W_{\la}W_{\lb}}{|W|^2}\frac{\omega_{FS}^n}{n!}}=\frac{\delta_{\la \lb}}{(n+1)!}
$$
$$
\int_{\mathbb{P}^n}{\frac{W_{\la}W_{\lb}W_{\lc}W_{\ld}}{|W|^4}\frac{\omega_{FS}^n}{n!}}=
\frac{\delta_{\la\lb}\delta_{\lc\ld} + \delta_{\la\ld}\delta_{\lc\lb}}{(n+2)!}
$$
\end{lemma}
Thus, using also the identity (\ref{detcur}) from the corollary above, we can write
\begin{eqnarray*}
&&\int_{X_s}{f^*R^{\det E}_G \delta \frac{W_i W_{\jbar}}{|W|^2}\; \frac{\omega_{FS}^n}{n!}}\\
&=& f^*R^{\det E}_G \delta \int_{X_s}{\frac{W_i W_{\jbar}}{|W|^2}\; \frac{\omega_{FS}^n}{n!}}\\
&=& \frac{f^*R^{\det E}_G}{r!} \delta \delta_{i\jbar}\\
&=& \int_{X_s}{\delta \, \delta_{i\jbar} \, c(\varphi) \;\frac{\omega_s^n}{n!}}
\end{eqnarray*}
Hence we can rewrite the expression for the curvature as
\begin{equation}
\label{break}
\Theta_{i\jbar s\sbar}(s) = \delta \int_{X_s}{((r+1)c(\varphi) \, \frac{W_i W_{\jbar}}{|W|^2} - \delta_{i\jbar} \, c(\varphi))\; \frac{\omega_{FS}^n}{n!}}.
\end{equation}
From this expression alone it it not yet possible to read off the Griffiths positivity of the direct image metric. We need a further ingredient that explains the structure of the geodesic curvature function $c(\varphi)$ which allows us to evaluate the integral. For this we go back to the elliptic equation (\ref{ell}). We can rewrite it as
$$
(\Box_{\omega_{FS}}-r)(rc(\varphi) - f^*R^{\det E}_G)=0.
$$  
 This means that the function $rc(\varphi) - f^*R^{\det E}_G$ is a eigenfunction to the smallest positive eigenvalue $r=n+1$ on the K\"ahler manifold $(\PR^n,\omega_{FS})$. But this functions are known. A base of eigenfunctions is given by
$$
\left( r\frac{W_{\beta} W_{\labar}}{|W^2|} - \delta_{\labar \beta} \right)_{\alpha,\beta=1}^r
$$
Hence we can write
\begin{equation}
\label{id}
rc(\varphi) - \int_{X_s}{r \, c(\varphi)\; \omega^n_{FS}} 
= r \sum_{\alpha,\beta=1}^r{\lam_{\labar \beta}\frac{W_{\beta} W_{\labar}}{|W^2|}} 
- \sum_{\alpha, \beta=1}^r \delta_{\labar \beta}\lam_{\labar \beta}
\end{equation}
for some constants $\lam_{\labar\beta} \in \C$. Now we would like to identify the function
$$
\tilde{c}(\varphi) = \sum_{\alpha,\beta=1}^r{\lam_{\labar \beta}\frac{W_{\beta} W_{\labar}}{|W^2|}}
$$
with $c(\varphi)$. We already know from the equation (\ref{id}) that they differ only by a constant. 
But if we choose 
$$
\lam_{\labar \alpha} = \Theta_{\alpha \labar}s\sbar(s)
$$
for $\alpha=1,\ldots, r$, we
have that
$$
\int_{X_s}{\tilde{c}(\varphi)\, \frac{\omega_{FS}^n}{n!}} 
= \tr \Theta_{\alpha \beta s \sbar} 
= \int_{X_s}{c(\varphi) \, \frac{\omega_{FS}^n}{n!}}.
$$
This shows that we can arrange the constants $\lam_{\labar\beta}$ in such a way such that on the fiber $X_s$ we indeed have
$$
c(\varphi) = \sum_{\alpha,\beta=1}^r{\lam_{\labar \beta}\frac{W_{\beta} W_{\labar}}{|W^2|}}.
$$
Plugging this into the equation (\ref{break}), we can proceed as
\begin{eqnarray*}
\Theta_{i\jbar s\sbar}(s) 
&=& \delta \int_{X_s}{((r+1)c(\varphi) \, \frac{W_i W_{\jbar}}{|W|^2} - \delta_{i\jbar} \, c(\varphi))\; \frac{\omega_{FS}^n}{n!}}\\
&=& \delta \int_{X_s}{((r+1) \lam_{\labar \beta}\frac{W_{\beta} W_{\labar}}{|W^2|} \, \frac{W_i W_{\jbar}}{|W|^2} - \delta_{i\jbar} \, \lam_{\labar \beta}\frac{W_{\beta} W_{\labar}}{|W^2|})\; \frac{\omega_{FS}^n}{n!}}\\
&=& \delta (r+1) \lam_{\labar \beta} \int_{X_s}{\frac{W_{\beta} W_{\labar} W_i  W_{\jbar}}{|W|^4}\; \frac{\omega_{FS}^n}{n!}} -
\delta \delta_{i\jbar} \, \lam_{\labar \beta} \int_{X_s}{\frac{W_{\beta} W_{\labar}}{|W^2|})\; \frac{\omega_{FS}^n}{n!}}\\
&=& \frac{\delta \lam_{\labar \beta}}{r!}\left( (\delta_{\beta \ovl{\la}}\delta_{i \jbar} + \delta_{\beta \jbar}\delta_{i \ovl{\la}}) - \delta_{\beta \labar}\delta_{i \jbar} \right)\\
&=& \lam_{i\jbar} \frac{\delta}{r!}.
\end{eqnarray*}
Now because the expression
$$
\sum_{i,j=1}^r{\Theta_{i\jbar s \sbar} \frac{W_i W_{\jbar}}{|W|^2}} = c(\varphi)
$$
is positive (because $(\Oh_E(1),h)$ was positive by assumption), we conclude that the $L^2$-metric on $f_*\Oh_E(1)\cong E$ is indeed Griffiths positive. This proves Theorem \ref{thm}.
\end{proof}

\newpage

\section{The relative K\"ahler-Ricci Flow}
\subsection{The flow on $K_{\X/S}^{-1}$}
First we recall the relative K\"ahler-Ricci flow for the bundle $K_{\X/S}^{-1}$ and a family of Fano-K\"ahler-Einstein manifolds $X_s$ as introduced in \cite{Be13}. As proved in \cite{Be13} this flow converges in $C^{\infty}$ to a K\"ahler-Einstein weight $\phi_{\infty}$ on any Fano-K\"ahler-Einstein manifold $X$ with $H^0(X,T_X)=0$. But using the recent developments concerning the K\"ahler-Ricci flow on Fano-K\"ahler-Einstein manifolds, we show that the condition about the automorphism group can be removed.

For this, we consider the absolute case of a Fano-K\"ahler-Einstein manifold X ($\dim X=n$) with a metric $h=e^{-\phi}$ on $K_X^{-1}$ of positive curvature. We write $\phi \in \mathcal{H}_X$, where $\mathcal{H}_X$ is the set of smooth and positive (weights for) metrics on $K_X^{-1}$. In this notation $h=e^{-\phi}$ is the pointwise norm squared of a local trivializing section 
$$
\dl/\dl z^1 \wedge \ldots \wedge \dl/\dl z^n
$$ 
of $K_X^{-1}$. But more globally, we can view $e^{-\phi}$ as a volume form on $X$ by regarding
$$
c_n dz^1\wedge \ldots \wedge dz^n\wedge dz^{\ovl{1}} \wedge \ldots \wedge dz^{\ovl{n}} e^{-\phi}.
$$
Therefore, we can define the \emph{canonical measure}
$$
\mu_{\phi}:= \frac{e^{-\phi}}{\int_X{e^{-\phi}}}.
$$
We have also the \emph{Monge-Amp\`{e}re measure} defined as
$$
\MA(\phi) := V^{-1}(dd^c \phi)^n,
$$
where $V=\int_X{(dd^c\phi)^n}$. These two measures only depend on the curvature form $\omega = dd^c \phi$ of $\phi$ and thus we write $\mu_{\phi}=\mu_{\omega}$. The K\"ahler-Einstein condition becomes
$$
V^{-1}\omega^n=\mu_{\omega}.
$$
The smooth function
\begin{equation}
\label{RicciPot}
u=u_{\omega}:= \log\left( \frac{V^{-1}\omega^n}{\mu_{\omega}}\right) = \left(\frac{\MA(\phi)}{\mu_{\phi}}\right)
\end{equation}
satisfies $dd^c u = \omega - \Ric(\omega)$ and $V^{-1}\int_X{e^{-u}\omega^n}=1$ and hence coincides with the normalized Ricci potential.

We consider the (normalized) K\"ahler-Ricci flow
\begin{equation}
\label{flowcur}
\dot{\omega}_t = \omega_t - \Ric(\omega_t)
\end{equation}
on the level of K\"ahler forms starting with some initial K\"ahler form $\omega_0 \in c_1(X)$ which is the curvature form of a positively curved hermitian metric $h_0=e^{-\phi_0}$ on $K_X^{-1}$.
This flow has a global solution and Perelman's estimates show that the normalized Ricci potential
 $$
 u_t := \log\left( \frac{V^{-1} \omega_t^n}{\mu_{\omega_t}}\right)
 $$
is uniformly bounded in $C^0$. This means that $\omega_t^n$ and $\mu_{\omega_t}$ remain uniformly comparable along the flow. The main result of \cite{PSSW09} is the following:

\begin{theorem*}
If there exists constants $C,\epsilon >0$ such that the Mabuchi K-energy $M(\omega_t)$ and the first eigenvalue positive eigenvalue $\lambda_1(\omega_t)$ of the $\dbar$-Laplacian with respect to $\omega_t$ on the space of smooth vector fields satisfy $M(\omega_t)\geq -C$ and $\lambda_1(\omega_t)\geq \epsilon$ along the flow, then $\omega_t$ converges at exponential rate in $C^{\infty}$ topology to a K\"ahler-Einstein metric. 
\end{theorem*}

As observed in \cite{CS12}, this combines with the main result of \cite{TZ11} to give
\begin{corollary*}
\label{Cor}
If $X$ is Fano and K\"ahler-Einstein, then the flow of K\"ahler forms $\omega_t$ according to (\ref{flowcur}) converges at exponential rate to a K\"ahler-Einstein metric. 
\end{corollary*}
\begin{proof}
Since the Mabuchi K-energy is minimized at a K\"ahler-Einstein metric, it is bounded from below. By the result of Tian and Zhu, there exists a path $g_t \in \Aut^0(X)$ such that $g_t^*\omega_t$ converges to a K\"ahler-Einstein metric. In particular, $\lambda_1(\omega_t)=\lambda_1(g_t^*\omega_t)$ is bounded from $0$ and we conclude by the previous result.
\end{proof}

Now we lift the flow (\ref{flowcur}) to level of hermitian metrics on $K_X^{-1}$. As introduced in \cite{Be13}, we consider the \emph{normalized flow} on the level of hermitian metric on $K_X^{-1}$ as follows:
\begin{equation}
\label{flowKx}
\dot{\phi}_t = \left( \frac{\MA(\phi_t)}{\mu_{\phi_t}}\right) = \log\left( \frac{V^{-1}(dd^c \phi_t)^n}{e^{-\phi_t}/\int_X{e^{-\phi_t}}}\right), \quad \phi_0=\phi.
\end{equation}
As already observed in \cite{Be13}, the fact that the two measures $\MA(\phi_t)$ and $\mu_{\phi_t}$ are both probability measures gives that the time derivative $\dot{\phi}_t$ coincides with the normalized Ricci potential $u_t$ defined by (\ref{RicciPot}). But now since $\omega_t$ converges exponentially fast to a K\"ahler-Einstein metric $\omega_{\infty}$,
$\dot{\phi}_t=u_t$ converges exponentially fast to $u_{\infty}=0$. But since the operator
$$
\omega \mapsto u_{\omega}=\left( \frac{V^{-1}\omega^n}{\mu_{\omega}}\right)
$$
is Lipschitz on bounded sets with respect to the $C^{\infty}$ topology, the exponential convergence can be integratet to imply that $t \mapsto \phi_t$ is Cauchy as $t \to \infty$. Hence $\phi_t$ converges to a hermitian metric $\phi_{\infty}$ which gives a K\"ahler-Einstein potential for $\omega_{\infty}$.

For the relative case of a family $f: \X \to S$ of Fano-K\"ahler-Einstein manifolds, we start with a metric $\phi$ on the anti-relative canonical bundle $K_{\X/S}^{-1}$. We define the flow fiberwise as in (\ref{flowKx}) and obtain convergence to a metric $\phi_{\infty}$ on $K_{\X/S}^{-1}$.

\subsection{The flow on $\Oh_E(r)$ - Definition and convergence}
Now we would like to run the flow on the bundle $\Oh_E(r)$ instead of $K_{\X/S}^{-1}$. Therefore, we first recall that the isomorphism
\begin{equation}
\label{iso1}
K_{\X/S}^{-1}\cong \Oh_E(r) \otimes f^*\det(E)^{-1}
\end{equation}
is canonical, because it comes from the natural relative Euler sequence
$$
0 \to \Omega_{\X/S} \to f^*E \otimes \Oh_E(-1) \to \Oh \to 0.
$$
Furthermore, the line bundle 
\begin{equation}
\label{iso2}
f_*(K_{\X/S}\otimes \Oh_E(r)) \cong \det E
\end{equation}
is fiberwise trivial and hence we have that
$$
H^0(X_s,K_{X_s}\otimes \Oh_{\PR(E_s^*)}(r)) \cong \det E_s
$$
is one-dimensional. 
Now we start with a hermitian metric on $\Oh(r)$ with fiberwise positive curvature. 
A local holomorphic (trivializing) section u of $\det E$ can be represented under the isomorphism (\ref{iso2}) by a holomorphic section of $K_{\X/S}\otimes \Oh(r)$ on the total space (after shrinking S), which are fiberwise holomorphic $n$-forms with values in $\Oh(r)$  (see also \cite{Sch12}). On the other hand, the pullback section $f^*u$ can locally on $\X$ be written as
$$
u = dz \otimes s
$$
with local trivializing sections $dz$ of $K_{\X/S}$ and $s$ of $\Oh_E(r)$ using the isomorphism
$$
f^*(\det E) \cong K_{\X/S} \otimes \Oh_E(r).
$$    
Of course, both descriptions are compatible with each other. This means that $dz$ and $s$ are only local, but $dz\otimes s$ is the local description of the global form $u$. Now let $e^{-\varphi}$ be the pointwise norm squared of $s$. The associated norm of 
$u_s \in H^0(X_s,K_{X_s}\otimes \Oh_{\PR(E^*)}(r))\cong \det E$ is given by the $L^2$-norm
$$
||u_s||^2 := \int_{X_s}{|u_s|^2e^{-\varphi}} := \int_{X_s}{c_n |s|^2e^{-\varphi}dz\wedge d\ovl{z}}.
$$
Now we write $e^{-\psi}:= ||u_s||^2$, which is a metric on $f^*{\det E}$. Hence 
$$
e^{-\phi} := e^{-\varphi}e^{\psi}
$$    
is a metric on $-K_{\X/S}$ corresponding to the trivializing section $(dz)^{-1}$. Now fiberwise, this is a volume form on $X_s$, in other words
$$
\mu_{\varphi}:=e^{-\varphi}/\int_{X_s}{|u_s|^2e^{-\varphi}}
$$
is a volume form on each fiber $X_s$, which does not depend on the choice of the local trivializing section $u$. Moreover, we have
\begin{proposition}
$\mu_{\varphi}$ is a probability measure.
\end{proposition}  
\begin{proof}
$$
\int_{X_s}{[e^{-\varphi}/\int_{X_s}{e^{-\varphi}|u_s|^2}]dz\wedge d\ovl{z}} 
= \frac{\int_{X_s}{e^{-\varphi}|s^2|dz\wedge d\ovl{z}}}{\int_{X_s}{|s|^2e^{-\varphi}dz \wedge d\ovl{z}}}=1.
$$
\end{proof}
Thus we define the flow of hermitian metrics $\varphi_t$ on $\Oh_E(r)$, which are all positive along the fibers $X_s$ by
\begin{equation}
\label{flow}
\dot{\varphi}_t = \log\left( \frac{\MA(\varphi_t)}{\mu_{\varphi}}\right), \quad \varphi_0=\varphi.
\end{equation}

\begin{theorem}
The relative K\"ahler-Ricci flow on $\Oh_E(r)$ as defined in (\ref{flow}) is convergent.
\end{theorem}
\begin{proof}
We need to show that the flow converges on each fiber $X_s$. For this we note that $K_{\X/S}^{-1}|_{X_s}=K_{X_s}^{-1}$ is canonical isomorphic to $\Oh_{\PR(E_s^*)}(r)\otimes f^*(\det E_{s})^{-1}$, hence $K_{X_s}^{-1}$ and $\Oh_{\PR(E_s^*)}(r)$ coincide up to a trivial line bundle on $X_s$. Thus we argue as before: By the previous proposition the measure $\mu_{\varphi}$ is a normalized volume form on $X_s$. Hence the time derivative $\dot{\varphi_t}$ coincides with the normalized Ricci potential. Now we use the same arguments as in the previous section to conclude that the flow converges in $C^{\infty}$ to a hermitian metric $\varphi_{\infty}$ on $\Oh_E(r)$ which is positive along the fibers.  
\end{proof}

\begin{remark}
As an observation we remark that the hermitian metric on $K_{\X/S}^{-1}$ given by the Monge-Amp\`{e}re measure 
$(dd^c\varphi_t)^n/n!$ splits in the limit into the product metric 
$$
e^{-\varphi_{\infty}}\cdot \left(\int{|u_s|^2 e^{-\varphi_{\infty}}}\right)^{-1},
$$ 
which is of course what we expect in the K\"ahler-Einstein limit.
\end{remark}
\begin{remark}
The construction works for any holomorphic vector bundle $E \to S$. In particular if $\det E$ is trivial, then the bundle $K_{\X/S}^{-1}$ is canonically isomorphic to $\Oh_E(r)$ and the flow (\ref{flow}) coincides with the normalized flow in \cite{Ber13}. 
\end{remark}

\subsection{The evolution equation for the geodesic curvature function}

Now we turn back to the Griffiths problem and start with a positive hermitian metric $h$ on $\Oh_E(1)$. Then we run the flow $(\ref{flow})$ on $\Oh_E(r)$ for the positive initial metric $h^r=e^{-\varphi_0}$. In order to apply Theorem \ref{thm}
to the limit metric, we are left to proof that the positivity of the geodesic curvature function $c(\varphi_t)$ is preserved under the flow (\ref{flow}). (Note that $e^{-\varphi_t}$ is now the metric on $\Oh_E(r)$.) Hence we need to study how it evolves along the flow. For this we can again assume that our base manifold $S$ is a one-dimensional disc. By using the computation from \cite{Ber13}, we get the following evolution equation for $c(\varphi)$:

\begin{proposition} 
On a fiber $X_s$ we have
\begin{equation}
\label{evolv}
\frac{\dl c(\varphi_t)}{\dl t} = -\Box_{\omega_t} c(\varphi_t) + rc(\varphi_t) + |A^{\varphi_t}_s|^2_{\omega_t}  
+\dl_s\dl_{\sbar} \log \int_{X_s}{|u_s|^2e^{-\varphi_t}}
\end{equation}
\end{proposition}
Note that in the limit we get indeed equation (\ref{ell}). 
For the last term on the right hand side, which is minus the curvature of 
$$
 \int{|u_s|^2 e^{-\varphi_t}},
$$
we can apply for instance the curvature formula from \cite{Be15}. But if $c(\varphi_t)$ is positive, this gives a negative contribution on the right hand side, which is an obstruction to apply the usual maximum principles for the heat equation.
Thus the question whether the positivity of $c(\varphi_0)$ is preserved or not remains open.

\end{document}